\documentclass{amsart}

\usepackage{ae}
\usepackage[all]{xy}
\usepackage{graphicx,amsfonts,amssymb,amsmath}
\usepackage{eurofont}
\usepackage{natbib}

\usepackage[latin1]{inputenc}
\usepackage[T1]{fontenc}
\usepackage[francais]{babel}
\usepackage[cyr]{aeguill}

\usepackage[OT2,T1]{fontenc}
\DeclareSymbolFont{cyrletters}{OT2}{wncyr}{m}{n}
\DeclareMathSymbol{\sha}{\mathalpha}{cyrletters}{"58}

 \newtheorem{thm}{Théorème}[section]
 \newtheorem*{thmp}{Théorème principal}

 \newtheorem{lem}[thm]{Lemme}
 \newtheorem{prop}[thm]{Proposition}
 \theoremstyle{definition}
 
 \theoremstyle{remark}
 
 \theoremstyle{remark}
 \newtheorem{rem}[thm]{\textbf{Remarque}}
 \numberwithin{equation}{subsection}

 \newcommand{\To}{\longrightarrow}

 \newcommand{\A}{\mathbb{A}}
 \newcommand{\Q}{\mathbb{Q}}
 \newcommand{\Z}{\mathbb{Z}}
\newcommand{\chapeau}{{\rlap{\smash{\hbox{\lower4pt\hbox{\hskip1pt$\widehat{\phantom{u}}$}}}}}\mbox{ }}

\usepackage[dvipdfm]{hyperref}
\begin{document}

\title{Astuce de Salberger et zéro-cycles sur certaines fibrations}

\author{ Yongqi LIANG  }

\address{Yongqi LIANG \newline
Département de Mathématiques, \newline Bâtiment 425,\newline
Université  Paris-sud 11,\newline  F-91405 Orsay,\newline
 France}

\email{yongqi.liang@math.u-psud.fr}

\thanks{\textit{Mots clés} : zéro-cycle de degré $1$, principe de Hasse, approximation faible,
obstruction de Brauer-Manin}

\thanks{\textit{Classification AMS} : 14G25 (11G35, 14D10)}

\date{\today.}

%\dedicatory{}

%%% ----------------------------------------------------------------------

%\begin{abstract}
%On démontre que l'obstruction de Brauer-Manin est la seule au
%principe de Hasse et à l'approximation faible pour les
%zéro-cycles sur certaines fibrations au-dessus
%d'une courbe lisse ou au-dessus de l'espace projectif,
%la nouveauté principale est que les
%hypothèses arithmétiques sont supposées seulement sur les
%fibres au-dessus d'un sous-ensemble hilbertien généralisé.
%\end{abstract}

%%% ----------------------------------------------------------------------
\maketitle
\small
\paragraph*{\textsc{Résumé}}
On démontre que l'obstruction de Brauer-Manin est la seule au
principe de Hasse et à l'approximation faible pour les
zéro-cycles sur certaines fibrations au-dessus
d'une courbe lisse ou au-dessus de l'espace projectif.
L'exactitude d'une suite de type global-local pour les groupes de Chow des zéro-cycles
est aussi établie pour ces variétés.
La nouveauté principale est que les
hypothèses arithmétiques sont supposées seulement sur les
fibres au-dessus d'un sous-ensemble hilbertien généralisé, de plus, on permet l'existence
des fibres géométriquement non intègres.

\medskip

\paragraph*{\textsc{Abstract}}
We prove that the Brauer-Manin obstruction is the only obstruction to the
Hasse principle and to weak approximation for zero-cycles on certain
fibrations over a smooth curve or over the projective space.
The exactness of a sequence of global-local type for Chow groups of zero-cycles is also
established for these varieties.
The principal
novelty is that the arithmetic hypotheses are supposed only on the fibers
over a generalized Hilbertian subset, moreover, we permit the existence of geometrically
non integral fibers.
\normalsize

\tableofcontents

\section*{Introduction}
\label{intro}

Soit $X$ une variété projective lisse et géométriquement intègre sur un
corps de nombre $k.$
On considère le principe de Hasse
pour les zéro-cycles de degré $1$ sur $X.$ On considère également, en un certain sens (précisé dans \S \ref{rappels}),
l'approximation faible/forte pour les zéro-cycles de degré $1$ sur $X.$
L'obstruction associée au groupe de Brauer $Br(X),$ dite de Brauer-Manin, est introduite par Manin dans son exposé \cite{Manin}
pour les points rationnels sur $X,$ et est étendue aux zéro-cycles par Colliot-Thélène dans \cite{CT95}.
Il a été conjecturé par Colliot-Thélène/Sansuc \cite{CTSansuc81}, Kato/Saito \cite{KatoSaito86},
et Colliot-Thélène \cite{CT95},
que la suite suivante (\textit{cf.} \S \ref{suiteE}) soit exacte
$$(E)\mbox{         }CH_0^{\chapeau}(X)\to CH_{0,\A}^\chapeau(X)\to Hom(Br(X),\Q/\Z),$$
qui signifie qu'une famille de zéro-cycles locaux orthogonale au groupe de Brauer $Br(X)$ de $X$
provient (modulo un entier donné) d'un zéro-cycle global.

Supposons toujours que $X$ admet un morphisme dominant $\pi:X\to B$ à fibre générique $X_\eta$ géométriquement intègre
sur le corps des fonctions $k(B),$ où $B$ est une variété projective lisse et géométriquement intègre.
Soit $H$ un sous-ensemble (à préciser ci-dessous) de points fermés de $B.$
On fait les hypothèses suivantes:

(HP/AF) le principe de Hasse/l'approximation faible (pour les points rationnels ou pour les zéro-cycles de degré $1$)
vaut pour la $k(\theta)$-variété $X_\theta$ pour tout point
fermé $\theta\in H\subset B;$

(\textsc{Abélienne-Scindée}) pour tout point $\theta$ de $B$ de codimension $1,$ la fibre $X_\theta$
possède une composante irréductible de multiplicité un,
dans le corps des fonctions de laquelle la fermeture algébrique de $k(\theta)$ est une extension
abélienne de $k(\theta).$

L'hypothèse (\textsc{Abélienne-Scindée}),
introduite par Colliot-Thélène/Skorobogatov/ Swinnerton-Dyer
\cite{CT-Sk-SD}, est automatiquement vérifiée si toutes les
fibres de $\pi$ sont géométriquement intègres.

En utilisant l'astuce de Salberger (\cite{Salberger}, \S 6), dans leur article \cite{CT-Sk-SD},
Colliot-Thélène/ Skorobogatov/Swinnerton-Dyer montrent que
l'obstruction de Brauer-Manin est la seule au principe de Hasse pour les zéro-cycles de degré $1$
sur $X$ si $B=\mathbb{P}^1$ et si $H$ est un ouvert dense de $B.$  Ce résultat
a été généralisé (au moins partiellement) dans deux directions différentes:
en généralisant la base $B$ et en affaiblissant l'hypothèse sur le sous-ensemble $H.$

\noindent- Initié par Colliot-Thélène \cite{CT99}, suivi par les travaux de Frossard \cite{Frossard} et de
van Hamel \cite{vanHamel}, on arrive au résultat récent de Wittenberg \cite{Wittenberg},
il montre la même assertion pour le cas où $B=C$ est une courbe lisse de genre quelconque en supposant
la finitude du groupe de Tate-Shafarevich $\sha(Jac(C))$ de sa jacobienne, avec $H$ un ouvert dense de $C;$
de plus, il montre l'exactitude de $(E)$ lorsque
$X_\eta$ est géométriquement rationnellement connexe.
Un énoncé similaire pour le cas où $B=\mathbb{P}^n$ avec $H$ un ouvert dense
est montré par l'auteur dans \cite{Liang2}, Théorème 3.5.

\noindent- Dans l'autre direction, afin d'appliquer le résultat aux solides de Poonen construits dans
\cite{Poonen}, l'auteur montre dans \cite{Liang1}
que l'obstruction de Brauer-Manin est la seule au principe de Hasse et à l'approximation faible pour
les zéro-cycles de degré $1$ sur $X,$ si l'on suppose que toutes les fibres de $\pi$ sont
géométriquement intègres,
si $B=C$ est une courbe lisse de groupe $\sha(Jac(C))$ fini, et si $H$
est un \textit{sous-ensemble hilbertien généralisé} (\textit{cf.} \S \ref{rappels}, en particulier
un ouvert dense est un tel sous-ensemble)
de $C.$

Le but de ce travail est de montrer le théorème principal suivant qui généralise simultanément
les résultats des deux types ci-dessus.

\begin{thmp}
Soit $\pi:X\to B$ un morphisme dominant entre des variétés projectives lisses et géométriquement intègres,
à fibre générique $X_\eta$ géométriquement intègre sur $k(B).$ On suppose
\begin{enumerate}
\setlength\itemindent{2.2cm}\item[(\textsc{Abélienne-Scindée})]
\setlength\itemindent{0cm}\item[ ]
pour tout point $\theta$ de $B$ de codimension $1,$
la fibre $X_\theta$ possède une composante irréductible de multiplicité un,
dans le corps des fonctions de laquelle la fermeture algébrique de $k(\theta)$ est une extension
abélienne de $k(\theta).$
\end{enumerate}

Soit $\textsf{Hil}$ un sous-ensemble hilbertien généralisé de $B.$
Supposons respectivement que

(1) pour tout point fermé $\theta\in \textsf{Hil},$
la fibre $X_\theta$ satisfait le principe de Hasse (pour les points rationnels ou pour les zéro-cycles de degré $1$);

(2) pour tout point fermé $\theta\in \textsf{Hil},$
la fibre $X_\theta$ satisfait l'approximation faible (pour les points rationnels ou pour les zéro-cycles de degré $1$);

(3) l'hypothèse (2) et de plus la fibre générique $X_\eta$ est géométriquement rationnellement connexe.

Alors, dans chacun des cas suivants

- la base $B=C$ est une courbe de groupe de Tate-Shafarevich $\sha(Jac(C))$ fini,

- la base $B=\mathbb{P}^n$ est l'espace projectif,

\noindent on a pour les zéro-cycles de degré $1$ sur $X$

(1) l'obstruction de Brauer-Manin est la seule au principe de Hasse;

(2) l'obstruction de Brauer-Manin est la seule à l'approximation faible;

(3) l'obstruction de Brauer-Manin est la seule à l'approximation forte, et de plus, la suite $(E)$
est exacte.
\end{thmp}

En comparant avec les résultats \cite[Théorèmes 1.3, 1.4]{Wittenberg} et \cite[Théorème 3.1]{Liang1},
dans ce travail l'hypothèse arithmétique est supposée au-dessus d'un sous-ensemble hilbertien généralisé
au lieu d'un ouvert dense,
plus de fibres qui ne satisfont pas le principe de Hasse ou l'approximation faible sont permise, d'où
on trouve une nouvelle application dans \S \ref{Fibreenchatelet}.
Ces deux anciens résultats sont basés sur le résultat de
Colliot-Thélène/Skorobogatov/Swinnerton-Dyer \cite[Théorème 4.1]{CT-Sk-SD}, là l'astuce de Salberger
a été utilisée pour sa démonstration. Dans ce article, on développe cette astuce afin de démontrer
le théorème ci-dessus pour le cas crucial où $B=\mathbb{P}^1.$

On montre d'abord dans \S \ref{Casparticulier1} le cas particulier
du théorème, hormis l'exactitude de $(E),$ où $B=\mathbb{P}^1$ sous l'hypothèse plus forte que (\textsc{Abélienne-Scindée}):

- toutes les fibres de $\pi$ sont géométriquement intègres.

\noindent Ensuite, dans \S \ref{Casparticulier2}, on adapte cette preuve à l'astuce de Salberger et on montre
le théorème pour le cas où $B=\mathbb{P}^1$ sous l'hypothèse (\textsc{Abélienne-Scindée}).
C'est une généralisation de \cite[Théorème 4.1]{CT-Sk-SD}.
À partir de ceci, en appliquant les méthodes de Wittenberg \cite{Wittenberg,WittenbergLNM}
(voir aussi \cite{Liang2}), on traite
le cas où $B=C$ est une courbe dans \S \ref{Cas1} \S \ref{suiteE-C} et le cas où $B=\mathbb{P}^n$ dans \S \ref{Cas2}
\S \ref{suiteE-P}.
Enfin, on discute l'application aux fibrations en surfaces de Châtelet dans \S \ref{Fibreenchatelet}.

%\linespread{1.3}

%%% ----------------------------------------------------------------------

%%\setcounter{section}{-1}

\section{Conventions et rappels}\label{rappels}
\indent\textbf{Conventions.}
Dans tout ce travail, une \emph{variété} désigne un schéma séparé de type fini sur un corps.
Une \textit{fibration} $\pi:X\to B$ signifie un morphisme dominant entre des variétés lisses et géométriquement
intègres à fibre générique géométriquement intègre.
Le corps de base $k$ est toujours un corps de nombres. Comme d'habitude, on note $\Omega_k$
(resp. $\Omega_k^\textmd{f},$ $\Omega_k^\infty$) l'ensemble
des places (resp. places finies, places archimédiennes) de $k.$ Pour chaque place $v\in\Omega_k,$
on note $k_v$ le corps local associé, et $k(v)$ le corps résiduel si $v$ est non-archimédienne.
On fixe une clôture algébrique $\bar{k}$ de $k,$ $\bar{k}_v$ de $k_v$ pour toute
$v\in\Omega_k.$ L'expression \og presque tout\fg\mbox{} signifie toujours
\og tout à l'exception d'un nombre fini\fg.
Soit $k'$ une extension finie de $k,$ pour un sous-ensemble $S$ de places de $k,$
on note $S\otimes_kk'$ l'ensemble des places de $k'$ au-dessus des places dans $S.$

\smallskip
\textbf{Accouplement de Brauer-Manin.}
Soit $X$ une variété projective lisse géométriquement intègre sur un corps $k,$ le composé de la
restriction et l'application d'évaluation définit un accouplement
\begin{equation*}
    \begin{array}{rcccl}
     \langle\cdot,\cdot\rangle_k:   Z_0(X) & \times & Br(X) & \to & Br(k),\\
          (\mbox{ }\sum_Pn_PP &,& b\mbox{ }) & \mapsto & \sum_Pn_Pcores_{k(P)/k}(b(P)),\\
    \end{array}
\end{equation*}
qui se factorise à travers l'équivalence rationnelle $\sim,$ où $Br(\cdot)=H^2_{\mbox{\scriptsize\'et}}(\cdot,\mathbb{G}_m)$
est le groupe de Brauer cohomologique.
Lorsque $k$ est un corps de nombres, on définit
l'\emph{accouplement de Brauer-Manin} pour les zéro-cycles:
\begin{equation*}
    \begin{array}{rcccl}
     \langle\cdot,\cdot \rangle _k:\prod_{v\in\Omega_k}Z_0(X_v) & \times & Br(X) & \to & \mathbb{Q}/\mathbb{Z},\\
         (\mbox{ }\{z_v\}_{v\in\Omega_k} & , & b\mbox{ }) & \mapsto & \sum_{v\in\Omega_k}inv_v(\langle z_v,b \rangle _{k_v}),\\
    \end{array}
\end{equation*}
où $inv_v:Br(k_v)\hookrightarrow\Q/\Z$ est l'invariant local en $v$ et où $X_v=X\times_kk_v.$

\smallskip
\textbf{Principe de Hasse et l'approximation faible/forte.}
On considère le principe local-global pour les zéro-cycles de degré $1.$ On dit que $X$ \emph{satisfait le principe
de Hasse} (resp. \emph{l'obstruction de Brauer-Manin est la seule au principe de Hasse pour $X$})
s'il existe un zéro-cycle global de degré
$1$ lorsqu'il existe une famille de zéro-cycles locaux de degré $1$ (resp. une famille de zéro-cycles locaux de degré $1$
orthogonale à $Br(X)$).
On dit que $X$ satisfait \emph{l'approximation faible (resp. forte) au niveau du groupe de Chow},
si pour tout entier strictement positif $m$, pour tout ensemble fini $S\subset\Omega_k$ (resp. pour
$S=\Omega_k$), et pour toute famille $\{z_v\}_{v\in\Omega_k}$ de zéro-cycles locaux de degré $1,$
il existe un zéro-cycle global $z=z_{m,S}$ de degré $1$ tel que $z$ et $z_v$ aient la même image dans $CH_0(X_v)/m$
pour toute $v\in S.$ On dit que \emph{l'obstruction de Brauer-Manin est la seule à l'approximation faible/forte
au niveau du groupe de Chow}, si l'on demande de plus $\{z_v\}_{v\in\Omega_k}\bot Br(X)$ dans la définition
précédente. Pour simplifier la terminologie, pour les zéro-cycles on dit simplement
\og \emph{l'approximation faible/forte}\fg\mbox{ } au lieu de \og l'approximation faible/forte
au niveau du groupe de Chow\fg.

\smallskip
\textbf{Sous-ensemble hilbertien généralisé.}
On rappelle la notion de \emph{sous-ensemble hilbertien généralisé}.
Soit $X$ une variété sur un corps $k,$ un sous-ensemble $\textsf{Hil}\subset X$
de points fermés est un \emph{sous-ensemble hilbertien généralisé} s'il existe un morphisme
étale fini $Z\buildrel\rho\over\To U\subset X$ avec $U$ un ouvert non-vide de $X$ et $Z$ intègre tel que
$\textsf{Hil}$ soit l'ensemble des points fermés $\theta$ de $U$ pour lesquels $\rho^{-1}(\theta)$ est connexe.
Si $X$ est une variété normale,
soit $\textsf{Hil}_i$ ($i=1,2$) un sous-ensemble hilbertien généralisé, on
peut trouver un sous-ensemble hilbertien généralisé $\textsf{Hil}\subset\textsf{Hil}_1\cap \textsf{Hil}_2,$
\textit{cf.} \cite{Liang1}, \S 1.2.

\smallskip
\textbf{Zéro-cycles.}
Soit $z=\sum n_iP_i$ un zéro-cycle de $X$ (avec les points fermés $P_i$ distincts).
On dit qu'il est \emph{séparable} si $n_i\in\{0,1,-1\}$ pour tout $i.$

Soit $\pi:X\to B$  un morphisme dominant, le zéro-cycle $z=\sum n_iP_i$
est dit \emph{déployé} (relativement à la fibration $\pi:X\to B$)
s'il existe un $k(P_i)$-point rationnel sur la fibre $X_{P_i}$
pour tout $i.$

Étant donné $P$ un point fermé de $X_v,$ on fixe un $k_v$-plongement $k_v(P)\To \bar{k}_v,$
$P$ est vu comme un point $k_v(P)$-rationnel de $X_v.$
On dit qu'un point fermé $Q$ de $X_v$ est \emph{suffisamment proche}  de $P$ (par rapport à un voisinage $U_P$ de $P$ dans
l'espace topologique $X_v(k_v(P))$), si $Q$ a corps résiduel $k_v(Q)=k_v(P)$ et si l'on peut choisir un
$k_v$-plongement $k_v(Q)\To \bar{k}_v$ tel que $Q,$ vu comme un $k_v(Q)$-point rationnel de $X_v,$ soit contenu dans $U_P.$
En étendant $\mathbb{Z}$-linéairement, cela a un sens de dire que $z'_v\in Z_0(X_v)$ est suffisamment proche de $z_v\in Z_0(X_v)$
(par rapport à
un système de voisinages des points qui apparaissent dans le support de $z_v$), en particulier
$deg(z'_v)=deg(z_v)$ si c'est le cas.
Wittenberg a montré le lemme suivant (en remarquant que si $z_v$ et $z'_v$
sont effectifs de degré $d$ et suffisamment proches,
ils définissent des $k_v$-points sur le produit symétrique $Sym^d(X),$ suffisamment proches
par rapport à la $k_v$-topologie).

\begin{lem}[Wittenberg \cite{Wittenberg}, Lemme 1.8]\label{lemWittenberg}
Soient $m$ un entier strictement positif et $X$ une variété lisse sur $k.$
Pour $v\in\Omega_k,$ soient $z_v$ et $z'_v$ des zéro-cycles de $X_v.$
Alors ils ont la même image dans $CH_0(X_v)/m$ lorsqu'ils sont suffisamment proches.
\end{lem}

\textbf{Groupe de Brauer vertical.}
Soit $\pi:X\to B$ un morphisme dominant entre des variétés lisses connexes.
Le groupe de Brauer $Br(X)$ est vu comme un sous-groupe de $Br(k(X)).$ La partie
verticale $Br_{vert}(X)\subset Br(X)$ consiste en les éléments de $Br(X)$ provenant
de $Br(k(B))$ par le morphisme $\pi^*:Br(k(B))\to Br(k(X))$ induit par $\pi.$
Le quotient $Br_{vert}(X)/\pi^*Br(B)$ est fini si $B$ est une courbe et si l'hypothèse
(\textsc{Abélienne-Scindée}) dans l'introduction est vérifiée, \textit{cf.} \cite{CT-Sk2000}, Lemme 3.1.

%%% ----------------------------------------------------------------------
\section{Cas particulier où les fibres sont géométriquement intègres et $B=\mathbb{P}^1$}\label{Casparticulier1}

Puisque la preuve complète du théorème principal avec $B=\mathbb{P}^1$ est longue, pour le confort du lecteur,
on la sépare en deux parties. La première partie \S \ref{Casparticulier1}
se concentre sur les problèmes que comment assurer que $\theta$ est un point fermé au lieu
d'un zéro-cycle effectif et comment contrôler $\theta$ tel que il soit contenu dans $\textsf{Hil}.$
La deuxième partie \S \ref{Casparticulier2} se concentre sur le traitement de (\textsc{Abélienne-Scindée}), \textit{i.e.}
l'astuce de Salberger, et sur comment l'adapter avec la première partie.
La deuxième partie peut être vue comme une interprétation géométrique de
l'astuce de Salberger (comparer avec la preuve du théorème 3.1 de \cite{CT-Sk-SD}).

\bigskip

Dans cette section, on montre un cas particulier du théorème principal (sauf l'exactitude de la suite $(E)$) où
$B=\mathbb{P}^1$ et on suppose que toutes les fibres sont géométriquement intègres au lieu
de l'hypothèse (\textsc{Abélienne-Scindée}),
les conclusions qu'on va montrer deviennent beaucoup plus simples, respectivement:

(1) le principe de Hasse pour les zéro-cycles de degré $1;$

(2) l'approximation faible pour les zéro-cycles de degré $1;$

(3) l'approximation forte pour les zéro-cycles de degré $1.$

\noindent Ce cas est aussi un cas particulier du théorème principal
de l'auteur \cite{Liang1}, la preuve présentée ici est plus compliquée que \cite{Liang1}, mais l'avantage est qu'on peut
adapter cette preuve à l'astuce de Salberger pour montrer une généralisation dans
\S \ref{Casparticulier2}. Cette preuve fait une partie essentielle de la preuve entière du théorème principal.

Premièrement, on admet la proposition suivante et montre le cas particulier considéré, ensuite, on montre la proposition.

\begin{prop}\label{keyprop1}
Soient $\pi:X\to\mathbb{P}^1$ une fibration à fibres géométriquement intègres et $D$ un sous-ensemble fini
de points fermés de $\mathbb{P}^1.$ Soit $\textsf{Hil}$ un sous-ensemble hilbertien généralisé de $\mathbb{P}^1.$
Supposons qu'il existe une famille
$\{z_v\}_{v\in\Omega_k}$ de zéro-cycles locaux de $X$ de degré $1.$

Alors, pour tout entier strictement positif $a$ et pour tout ensemble fini $S\subset\Omega_k,$
il existe les données suivantes:

(a) pour chaque $v\in S,$ un zéro-cycle effectif $z^2_v\in Z_0(X_v)$ tel que
$z_v-z_v^2$ soit $a$-divisible dans $Z_0(X_v);$ et pour tout tel $z^2_v,$ un
zéro-cycle effectif $\tau_v\in Z_0(X_v)$ tel que
$\pi_*(\tau_v)$ soit séparable à support hors de $D,$
et tel que $\tau_v$ soit suffisamment proche de $z^2_v;$

(b) un point fermé  $\theta \in\textsf{Hil}$ de degré $d\equiv1(mod\mbox{ }a)$ tel que
$\theta$ soit déployé localement partout,  tel que comme zéro-cycle $\theta$ soit suffisamment proche de $\pi_*(\tau_v)$
pour toute $v\in S,$ a fortiori $\theta$ et $\pi_*(z_v)$ sont rationnellement équivalents modulo $a,$
\textit{i.e.} ils ont la même image dans $CH_0(\mathbb{P}_v^1)/a$ pour toute $v\in\Omega_k.$
\end{prop}

\begin{proof}[Démonstration du théorème principal sous les hypothèses au début de \S \ref{Casparticulier1}.]
On part d'une famille $\{z_v\}_{v\in\Omega_k}$ de zéro-cycles de degré $1$ sur $X,$
la proposition \ref{keyprop1} donne un point fermé $\theta\in\textsf{Hil}$ qui satisfait (a) et (b)
de la proposition.
Si $X_\theta$ satisfait le principe de Hasse (pour les points rationnels ou
pour les zéro-cycles de degré $1$), il existe un zéro-cycle global $z_\theta$ de degré $1$
sur la $k(\theta)$-variété $X_\theta,$ c'est un zéro-cycle de degré $d\equiv1(mod\mbox{ }a)$ sur $X.$
Si l'on prend pour $a$ le degré $d_Q$ d'un certain point fermé $Q$ de $X,$
le zéro-cycle $z=z_\theta-hQ$ est alors de degré $1$ sur $X$ pour un certain entier convenable $h,$
ceci montre (1).
On fixe un entier strictement positif $m$ et un sous-ensemble fini $S$ de places de $k.$
On suppose que la $k(\theta)$-variété $X_\theta$ satisfait l'approximation faible
(pour les points rationnels ou pour les zéro-cycles de degré $1$),
d'après le théorème des fonctions implicites et le lemme \ref{lemWittenberg}, on peut choisir
le zéro-cycle $z_\theta$ ci-dessus tel que
$z_\theta$ et $\tau_v$ ont la même image dans $CH_0(X_v)/m$ pour toute $v\in S.$
De l'autre côté, grâce au lemme \ref{lemWittenberg}, la proposition \ref{keyprop1}(a) implique que
$\tau_v$ et $z_v^2$ ont la même image dans $CH_0(X_v)/a.$
Si l'on prend $a$ un multiple de $d_Qm,$
les zéro-cycles $z$ et $z_v$ ont la même image dans $CH_0(X_v)/m$ pour toute $v\in S,$
ceci montre (2). Puisque $X_\eta$ est géométriquement rationnellement connexe, d'après le corollaire 2.2
de Wittenberg \cite{Wittenberg}, l'application
$CH_0(X_v)\to CH_0(\mathbb{P}^1_v)$ est injective pour presque toute $v.$ Quitte à augmenter $S,$
on peut supposer l'injectivité pour toute $v\notin S.$ Pour une telle $v,$ on a
$\theta\sim\pi_*(z_v)+ac_v$ pour un certain $c_v\in Z_0(\mathbb{P}^1_v).$
Si l'on prend $a=a'm$ tel que $a'$ soit un multiple de l'indice de la fibre générique
$X_\eta,$ le zéro-cycle $ac_v$ s'écrit comme $m\pi_*(z^0_v)$ pour un certain $z^0_v\in Z_0(X_v),$
\textit{cf.} \cite{Wittenberg}, Lemme 2.4. D'où $z_\theta\sim z_v+mz_v^0$ sur $X_v,$ ceci
montre (3).
\end{proof}

\begin{lem}\label{anydeglocalfield}
Soit $D$ un ensemble fini de points fermés de $\mathbb{P}^1_k,$ où $k$ est un corps local non-archimédien.
Alors, pour tout nombre entier strictement positif $n,$ il existe un point fermé de $\mathbb{P}^1\setminus D$ de degré
$n.$
\end{lem}

\begin{proof}
Comme $k$ est un corps
local non-archimédien, il existe alors un polynôme irréductible de degré $n,$ qui va définir un point fermé de
$Spec(k[T])=\mathbb{A}^1\subset\mathbb{P}^1$ de degré $n.$ Il y a un nombre infini de tels polynômes, par exemple
les polynômes d'Eisenstein, on peut donc le choisir
tel que le point fermé associé soit en dehors de l'ensemble fini $D.$
\end{proof}

\begin{proof}[Démonstration de la proposition \ref{keyprop1}.]
L'idée de cette démonstration provient de la preuve du théorème 1.3 d'Ekedahl \cite{Ekedahl}
et de la preuve de la proposition 3.2.1 de Harari \cite{Harari}.

On note $K=k(\mathbb{P}^1)$ et $K_Z=k(Z)$ l'extension finie de $K$ associée au revêtement fini $Z\to \mathbb{P}^1$
définissant $\textsf{Hil},$ qui est étale au-dessus d'un ouvert dense $U\subset\mathbb{P}^1.$
On prend $K'$ la clôture galoienne de $K_Z$ dans $\bar{K}$ une clôture algébrique de $K$ fixée au début.
Soit $Z'$ la courbe intègre normale projective de corps des fonctions $K'$ avec les morphismes finis
$Z'\to Z\to \mathbb{P}^1$ associés aux extensions des corps $K\subset K_Z\subset K',$
on note $U'$ l'image réciproque de $U$ dans $Z'.$
On note $k'$ la fermeture algébrique de $k$ dans $K',$ l'extension $k'/k$ est finie, on note $h$ son degré.

\begin{lem}\label{lemgeoint}
Les donnés ci-dessus satisfont: (quitte à restreindre $U$ et $U'$ si nécessaire)

(1) le $k$-morphisme $U'\To U$ est étale fini surjectif galoisien,

(2) $U'$ est une courbe géométriquement intègre au-dessus de $k',$

(3) le diagramme suivant est commutatif, où $U'\To U_{k'}$ est un $k'$-morphisme.
\SelectTips{eu}{12}$$\xymatrix@C=20pt @R=14pt{
U'\ar[d]\ar[rd]&\\
U&U_{k'}=U\times_kk'\ar[l]
}$$

(4) les fibres de $\pi:X\to\mathbb{P}^1$ au-dessus de $U$ sont lisses et géométriquement intègres.
\end{lem}

\begin{proof}
Comme l'extension $k'/k$ est finie et $car(k)=0,$ on écrit $k'=k(e)$ avec $e\in k'.$
Son image par le $k$-plongement $\iota:k'\to K'$ est une fonction  rationnelle $\iota(e)$ sur la courbe projective $Z'.$
On note $P$ l'ensemble fini des pôles de $\iota(e).$ Quitte à restreindre $U$ et $U',$
on peut supposer que $U'\cap P=\emptyset$
et que $U'$ se surjecte sur $U$ par le morphisme
$Z'\to\mathbb{P}^1.$ Le morphisme des $k$-algèbres $k'\to \mathcal{O}_{Z'}(U')\subset K'$ est alors bien défini, qui  donne
un $k$-morphisme $U'\to Spec(k').$
On peut supposer de plus que
les fibres de $\pi:X\to\mathbb{P}^1$ au-dessus de $U$ sont lisses et géométriquement intègres,
et que $U'\to U$ est étale galoisien (en enlevant l'orbite
de $P$ sous l'action de $Gal(K'/K)$ et les points ramifiés),
pour plus de détails sur la théorie de Galois pour une courbe
algébrique intègre normale \textit{cf.} \cite{Szamuely}, Chapitre 4.
La courbe $U'$ est géométriquement intègre sur $k',$ car $k'$ est algébriquement fermé dans $K'=k'(U').$
Le $k'$-morphisme canonique $U'\To U_{k'}=U\times_kk'$ satisfait le diagramme commutatif dans (3).
\end{proof}

Pour trouver un point fermé ${\theta}\in\textsf{Hil}\subset\mathbb{P}^1,$
il suffit de trouver  ${\theta}$
tel que la fibre de $U'\to U$ en ${\theta}$ soit connexe.

Quitte à augmenter $D,$ on peut supposer que l'ensemble fini $D$ contient
tout point fermé $\theta$ de $\mathbb{P}^1$ dont sa fibre $X_\theta$ n'est pas lisse.

On note $G=Gal(K'/K),$ c'est le groupe de Galois du revêtement fini étale $U'\to U,$
le revêtement fini étale $U'\to U_{k'}$ est aussi galoisien, de groupe noté par $H,$ qui est un sous-groupe de $G.$
On définit $I\subset\Omega_k$ comme l'ensemble des places de $k$ qui sont totalement décomposées
dans (la clôture galoisienne de) $k',$ c'est un ensemble infini d'après le théorème de \v{C}ebotarev.

On étend le $k'$-revêtement fini étale galoisien $U'\to U_{k'}$ à un modèle entier
$\mathcal{U}'\to\mathcal{V}$ au-dessus de $O_{k',S'_1}$ qui reste
un revêtement fini étale  galoisien de groupe $H,$ où $S'_1\subset\Omega_{k'}$
est un ensemble fini, \textit{cf.} le théorème 2.1 de \cite{Margaux} et (8.4.4) de \cite{EGAIV}.
En augmentant $S'_1$ si nécessaire, on peut supposer que $S'_1=S_1\otimes_kk'$ pour un certain ensemble
fini $S_1\subset\Omega_k$ et qu'il existe un modèle
$\mathcal{U}$ de $U$ sur $O_{k,S_1}$ satisfaisant le diagramme commutatif suivant, dont $\mathcal{U}'\to\mathcal{V}$ est un
$O_{k',S'_1}$-morphisme et les autres deux flèches sont des $O_{k,S_1}$-morphismes.
\SelectTips{eu}{12}$$\xymatrix@C=20pt @R=14pt{
\mathcal{U}'\ar[d]_G\ar[rd]^H &\\
\mathcal{U}&\mathcal{V}\ar[l]
}$$

L'estimation de Lang-Weil avec le lemme de Hensel  donne un sous-ensemble fini $S_2$ de $\Omega_k$ tel que
pour tout point fermé $\theta$ de $\mathbb{P}^1$ tel que $X_\theta$ soit lisse, on ait
$X_\theta(k(\theta)_w)\neq\emptyset$ pour toute $w\in (\Omega_k\setminus S_2)\otimes_kk(\theta),$
\textit{cf.} \cite{Liang1}, Lemme 3.3.

En augmentant $S$ si nécessaire, on peut supposer que $S\supset S_1\cup S_2\cup\Omega_\infty.$
Quitte à remplacer l'entier $a$ par $a[k':k],$ on peut supposer que $a$ est un multiple de $[k':k].$

On fixe un point fermé $z_0\in Z_0(X)$ de degré $d_0$ tel que $y_0=\pi_*(z_0)$ soit
séparable et à support en dehors de $D.$
On choisit un $k$-point noté par $\infty$ de $\mathbb{P}^1$ hors de $D\cup supp(y_0).$
On part d'une famille de zéro-cycles $\{z_v\}_{v\in\Omega_k}.$ Pour $v\in S,$
on écrit $z_v=z_v^+-z_v^-$ où $z_v^+$ et $z_v^-$ sont effectifs à supports disjoints.
On pose $z_v^1=z_v+ad_0z_v^-=z_v^++(ad_0-1)z_v^-$ de degré $\equiv 1(mod\mbox{ }ad_0),$ mais ils ne sont pas nécessairement de
même degré quand $v\in S$ varie.
On leur ajoute un multiple convenable de $az_0$ pour chaque $v,$ et on obtient $z_v^2$ effectif de même degré assez grand
$d\equiv 1(mod\mbox{ }ad_0).$
D'après le lemme 1.2 de \cite{Liang2},
on bouge $z_v^2$ un peu et obtient $\tau_v\in Z_0(X_v)$ effectif suffisamment proche de  $z_v^2$
tel que $\pi_*(\tau_v)$ soit séparable à support en dehors de $D\cup\{supp(y_0)\}\cup\{\infty\}.$

On trouve une fonction $f_v\in k_v(\mathbb{P}^1)^*/k_v^*$
telle que $div_{\mathbb{P}^1_v}(f_v)=\pi_*(\tau_v)-d\infty,$ pour toute $v\in S.$

\begin{lem}\label{E}
Soit $\textbf{E}$ l'ensemble fini des classes de conjugaison du groupe $H.$ Alors, il existe une injection
$$\gamma:\textbf{E}\to((\Omega_k\setminus S)\cap I)\otimes_kk'$$ qui satisfait les conditions suivantes,

\noindent- pour tout $c\in\textbf{E}$ il existe un point de corps résiduel fini
$\bar{x}_c\in\mathcal{V}(k'(\gamma(c)))$ tel que le Frobenius associé $Frob_{\bar{x}_c}$ soit contenu dans
la classe $c;$

\noindent- les places $v_{c_1},v_{c_2}\in\Omega_k$ sont distincts si $c_1\neq c_2\in\textbf{E},$ où $v_c$
est la place de $k$ au-dessous de $\gamma(c)\in\Omega_{k'}$ pour $c\in\textbf{E}.$

De plus, si l'on note $v=v_c$ et $w'=\gamma(c),$ les extensions
$k'_{w'}/k_v$ et $k'(w')/k(v)$ sont triviales, donc
$\mathcal{U}(k(v_c))=\mathcal{V}(k'(\gamma(c)))$ et $U(k_{v_c})=U_{k'}(k'_{\gamma(c)})$ pour tout $c\in\textbf{E}.$
\end{lem}

\begin{proof}
On remarque que $\mathcal{U}'\to\mathcal{V}$ est un revêtement galoisien de groupe $H$ où la fibre générique $U'$ de
$\mathcal{U}'\to Spec(O_{k',S'_1})$ est une variété géométriquement intègre au-dessus de $k',$
le théorème de densité de \v{C}ebotarev géométrique (\cite{Ekedahl}, Lemme 1.2)
donne l'injection $\gamma$ vérifiant la première condition. L'infinité de l'ensemble $I$
assure que la deuxième condition peut simultanément être vérifiée. La dernière assertion provient
de la construction de $I.$
\end{proof}

D'après le lemme de Hensel, pour chaque $c\in\textbf{E},$ le point
$\bar{x}_c$ se relève en un point $x_c\in\mathcal{V}(O_{w'})\subset U_{k'}(k'_{w'})$
où $O_{w'}$ est l'anneau d'entiers du corps local $k'_{w'}$ avec $w'=\gamma(c).$
D'après le lemme \ref{anydeglocalfield},
on trouve alors un point fermé $x'_c$ de $U_{v_c}$ de degré $d-1$
en dehors de $D\cup supp(y_0)\cup \{\infty,x_c\}.$ On a
$x_c+x'_c\sim d\infty\in Z_0(\mathbb{P}^1_{v_c})$ et
on obtient une fonction $f_{v_c}\in k_{v_c}(\mathbb{P}^1)^*/k_{v_c}^*$
telle que $div_{\mathbb{P}^1_{v_c}}(f_{v_c})=(x_c+x'_c)-d\infty\in Z_0(\mathbb{P}^1_{v_c})$
pour tout $c\in \textbf{E}.$

De même, on prend $v_0\in\Omega_k\setminus S\setminus \{v_c,c\in\textbf{E}\}$ et obtient
un point fermé $x_{v_0}$ de $U_{v_0}\subset\mathbb{P}^1_{v_0}$ de degré $d$ en dehors de
$D\cup supp(y_0)\cup \{\infty\}$ et une fonction $f_{v_0}\in k_{v_0}(\mathbb{P}^1)^*/k_{v_0}^*$ telle que
$div_{\mathbb{P}^1_{v_0}}(f_{v_0})=x_{v_0}-d\infty\in Z_0(\mathbb{P}^1_{v_0}).$

D'après le théorème de Riemann-Roch $\Gamma(\mathbb{P}^1,\mathcal{O}_{\mathbb{P}^1}(d\infty))$
est un espace vectoriel de dimension $d.$
On applique l'approximation faible pour $\mathbb{P}^{d-1},$
on trouve une fonction $f\in k(\mathbb{P}^1)^*$ qui est suffisamment
proche des $f_v$ pour $v\in S\cup\{v_c,c\in\textbf{E}\}\cup\{v_0\}.$
On a alors $div_{\mathbb{P}^1}(f)={\theta}-d\infty$ avec
${\theta}$ un zéro-cycle effectif séparable à support en dehors de $D\cup supp(y_0)\cup\{\infty\},$
de plus,

(i) ${\theta}$ est suffisamment proche de $\pi_*(\tau_v)$ pour $v\in S,$

(ii) ${\theta}$ est suffisamment proche de $x_c+x'_c$ pour $c\in \textbf{E},$

(iii) ${\theta}$ est suffisamment proche de $x_{v_0}.$

Le zéro-cycle ${\theta}$ est en fait un point fermé de $U\subset\mathbb{P}^1$ de degré $d$ par (iii).

Le zéro-cycle ${\theta}$ est  déployé localement partout. En fait,
pour les places dans $(\Omega_k\setminus S)\otimes_kk({\theta})$ l'assertion
suit de l'estimation de Lang-Weil mentionnée précédemment, pour les places dans $S\otimes_kk({\theta})$
l'assertion suit du théorème des fonctions implicites.

Pour conclure, il reste à montrer que la fibre de $U'\to U$ au point ${\theta}$ est connexe.

On note $L=k({\theta}),$ on a alors $[L:k]=d.$
Le point fermé ${\theta},$ vu comme un $L$-point rationnel, est suffisamment proche de $x_c+x'_c$ pour
$c\in\textbf{E},$ ceci implique qu'il existe
$w\in\Omega_{L}$ au-dessus de $v_c$ tel que $L_w/k_{v_c}$ soit une extension triviale et
l'image de ${\theta}$ par l'application $U(L)\to U(L_w)$
est suffisamment proche de $x_c\in \mathcal{U}(O_w)\subset U(L_w)=U(k_{v_c}).$
Donc ${\theta}$ est un point entier de $U$ (pour le modèle $\mathcal{U}$) dont
la réduction modulo $w$ est exactement
$\bar{x}_c\in\mathcal{U}(L(w))=\mathcal{U}(k(v_c))=\mathcal{V}(k'(\gamma(c))),$ où la deuxième égalité
provient du lemme \ref{E}.

On considère le revêtement (fini étale) $\phi:U_{k'}\to U.$ Le point fermé ${\theta}$ de $U$
donne un zéro-cycle $\phi^*({\theta})=Spec(L)\times_UU_{k'}\simeq Spec(L\otimes_kk')$ de $U_{k'}$ de degré $d.$
Comme $d\equiv1(mod\mbox{ }a),$ $d$ est premier à $[k':k],$
l'algèbre étale $L'=L\otimes_kk'$ reste alors un corps,
le zéro-cycle ${\theta}'=\phi^*({\theta})$ est donc un point fermé
de $U_{k'}$ de corps résiduel $k'({\theta}')=L'.$ En notant
$w'=\gamma(c)|v_c,$ on rappelle que dans le lemme \ref{E} on sait $k'_{w'}=k_{v_c},$  $k'(w')=k(v_c),$
$\mathcal{U}(k(v_c))=\mathcal{V}(k'(w')),$ et $U(k_{v_c})=U_{k'}(k'_{w'}).$ Le point ${\theta}'\in U_{k'}(L'_{\lambda})$ définit
en fait un point entier de $U_{k'}$ (pour le modèle $\mathcal{V}$) de réduction modulo $\lambda$ exactement
$\bar{x}_c\in\mathcal{V}(L'(\lambda))=\mathcal{V}(k'(w')),$
où $\lambda|v_c$ est une des places de $L'$ au-dessus de $w'=\gamma(c)\in\Omega_{k'}$ et au-dessus de $w\in\Omega_L$ à la fois
($\lambda$ est alors totalement décomposée  au-dessus de $v_c$).
En considérant le revêtement galoisien $\mathcal{U}'\to\mathcal{V}$ du groupe $H,$
comme $\theta'$ est un point de réduction $\bar{x}_c,$
la classe de conjugaison dans $H$ de l'automorphisme de Frobenius
$Frob_{\bar{x}_c}$ rencontre l'image de l'application
$Gal(\bar{L'}/L')=Gal(\overline{k'(\theta')}/k'(\theta'))\to H,$ cette dernière flèche est induite par
le point fermé ${\theta}'$ de $U_{k'}$
via le choix d'un relèvement du composé $Spec(\bar{L'})\to {\theta}'=Spec(L')\to U_{k'}$ à $U'.$
Ceci vaut pour tout $c\in\textbf{E}.$
L'application $Gal(\bar{L'}/L')\to H$ est donc surjective d'après un argument de la théorie des groupes finis,
\textit{cf.} \cite{Ekedahl}, Lemme 1.1.
La fibre de $U'\to U$ en ${\theta}$ est exactement la fibre de
$U'\to U_{k'}$ en ${\theta}'=\phi^{-1}({\theta}),$ elle est
alors connexe, ainsi ${\theta}\in \textsf{Hil}.$
\end{proof}

\begin{rem}\label{whynotC}
La méthode de la preuve présentée ici ne fonctionne que pour $\mathbb{P}^1.$
Si l'on part d'une courbe en bas $C$ de genre quelconque, on va
obtenir ${\theta}$ un zéro-cycle global séparable de $C$ qui n'est pas
nécessairement un point fermé. On écrit ${\theta}=\sum_j {\theta}_j$
où ${\theta}_j\simeq Spec(L_j)$ sont des points fermés districts de $C.$
On va trouver que $H$ est engendré par les images de $Gal(\bar{L}_j/L_j)\to H.$ Ceci ne suffit pas pour déduire que
la fibre de $Z'\to C$ en chaque ${\theta}_j$ est connexe.
\end{rem}

%%% ----------------------------------------------------------------------

\section{Cas particulier où (\textsc{Abélienne-Scindée}) est vérifiée et $B=\mathbb{P}^1$}\label{Casparticulier2}

Dans cette section, on montre un cas particulier du théorème principal (sauf l'exactitude de la suite $(E)$) où
$B=\mathbb{P}^1$ et on fait l'hypothèse (\textsc{Abélienne-Scindée}). Dans ce cas,
l'obstruction de Brauer-Manin associée au sous-groupe
$Br_{vert}(X)$ suffit,
les conclusions qu'on va montrer deviennent respectivement:
pour les zéro-cycles de degré $1$ sur $X$

(1) l'obstruction de Brauer-Manin associée au groupe $Br_{vert}(X)$ est la seule au principe de Hasse;

(2) l'obstruction de Brauer-Manin associée au groupe $Br_{vert}(X)$ est la seule à l'approximation faible;

(3) l'obstruction de Brauer-Manin associée au groupe $Br_{vert}(X)$ est la seule à l'approximation forte.

\noindent Ce cas est une généralisation du théorème 4.1 de
Colliot-Thélène/Skorobogatov/ Swinnerton-Dyer \cite{CT-Sk-SD} au sens que
$\textsf{Hil}$ est un sous-ensemble hilbertien généralisé au lieu d'un ouvert dense de $\mathbb{P}^1.$
La preuve suit la méthode utilisée dans \cite{CT-Sk-SD}, l'outil principal est l'astuce de Salberger,
à laquelle l'argument de \S \ref{Casparticulier1} est adapté.
Afin de baisser la difficulté de la lecture, au lieu de montrer ce cas directement, on le divise
en deux étapes: \S \ref{Casparticulier1} et \S \ref{Casparticulier2}, la preuve ci-dessous
ne répète pas l'argument pour la partie précédente \S \ref{Casparticulier1}.

Du même argument que \S \ref{Casparticulier1},
on se ramène à la proposition suivante, qui
joue le rôle de la proposition \ref{keyprop1}.
Dans le reste de cette section, on montre la proposition.

\begin{prop}\label{keyprop2}
Soit $\pi:X\to\mathbb{P}^1$ une fibration qui satisfait l'hypothèse

\noindent(\textsc{Abélienne-Scindée}) pour tout point $\theta$ de $\mathbb{P}^1$ de codimension $1,$
la fibre $X_\theta$ possède une composante irréductible de multiplicité un,
dans le corps des fonctions de laquelle la fermeture algébrique de $k(\theta)$ est une extension
abélienne de $k(\theta).$

Soient $\textsf{Hil}$ un sous-ensemble hilbertien généralisé de $\mathbb{P}^1$ et
$D$ un ensemble fini de points fermés de $\mathbb{P}^1.$
Supposons qu'il existe une famille
$\{z_v\}_{v\in\Omega_k}$ de zéro-cycles de $X$ de degré $1$ orthogonale au groupe
$Br_{vert}(X).$

Alors, pour tout entier strictement positif $a$ et pour tout ensemble fini $S\subset\Omega_k,$
il existe les données suivantes:

(a) pour chaque $v\in S,$ un zéro-cycle effectif $z^2_v\in Z_0(X_v)$ tel que
$z_v-z_v^2$ soit $a$-divisible dans $Z_0(X_v);$ et pour tout tel $z^2_v,$ un zéro-cycle effectif $\tau_v\in Z_0(X_v)$ tel que
$\pi_*(\tau_v)$ soit séparable à support hors de $D,$
et tel que $\tau_v$ soit suffisamment proche de $z^2_v;$

(b) un point fermé  $\theta \in\textsf{Hil}$ de degré $d\equiv1(mod\mbox{ }a)$ tel que
$\theta$ soit déployé localement partout,  tel que comme zéro-cycle $\theta$ soit suffisamment proche de $\pi_*(\tau_v)$
pour toute $v\in S,$ a fortiori $\theta$ et $\pi_*(z_v)$ sont rationnellement équivalents modulo $a,$
\textit{i.e.} ils ont la même image dans $CH_0(\mathbb{P}_v^1)/a$ pour toute $v\in\Omega_k.$
\end{prop}

\begin{proof}
On considère la fibration $\pi:X\to\mathbb{P}^1.$ Soit $U$ un ouvert dense de $\mathbb{P}^1$ tel que toute fibre $X_\theta$ au-dessus
d'un point $\theta\in U$ est lisse et géométriquement intègre.
On note $D_0$ l'ensemble des points au-dessus desquels les fibres sont non lisses ou géométriquement non
intègres,
on écrit $D_0=\{P_1,\ldots,P_i,\ldots,P_n\}$ où $P_i$ est un point fermé de $\mathbb{P}^1$ de corps résiduel $k_i=k(P_i).$
On choisit un $k$-point de $\mathbb{P}^1\setminus D_0$ noté par $\infty$ tel que la fibre $X_\infty$ soit lisse et
géométriquement intègre, alors $D_0\subset\mathbb{A}^1=\mathbb{P}^1\setminus\infty,$
quitte à restreindre $U,$ on peut supposer que $U\subset\mathbb{A}^1.$
Chaque point fermé $P_i$ donne un point $k_i$-rationnel $e_i\in k_i=\mathbb{A}^1(k_i),$
on note $g'_i=t-e_i\in k_i(t)=k_i(\mathbb{P}^1)$ et $g_i=N_{k_i(\mathbb{P}^1)/k(\mathbb{P}^1)}(g'_i)\in k(t)=k(\mathbb{P}^1).$
Le point $P_i$ est localement défini par $g_i.$

Soient $\mathcal{P}^1,$ $\mathcal{X},$ et $\mathcal{U}$ des modèles entiers lisses
sur $Spec(O_{k,S})$ de $\mathbb{P}^1$, $X,$
et $U,$ pour un sous-ensemble fini $S\subset\Omega_k$ suffisamment grand tel qu'il existe un $O_{k,S}$-morphisme
projectif $\Pi:\mathcal{X}\to\mathcal{P}^1$ dont la fibre générique au-dessus de $Spec(k)$ est $\pi.$
De plus, on peut supposer que $g_i\in O_{k,S}[\mathbb{A}^1].$
On note, pour tout $1\leqslant i\leqslant n,$ $T_i$ l'adhérence de Zariski
de $P_i$ dans $\mathcal{P}^1,$ c'est aussi l'adhérence de Zariski dans $\mathcal{P}^1$ du sous-schéma fermé de
$\mathbb{A}^1_{O_{k,S}}$ défini par $g_i=0.$ Quitte à augmenter $S,$ on peut supposer que les points schématiques de $\mathcal{P}^1$,
au-dessus desquels les fibres de $\Pi$ sont géométriquement non intègres, sont tous contenus dans $T=\bigcup T_i,$
que $T_i\cap T_j=\emptyset$ si $i\neq j,$ et que $T_i$ est étale sur $Spec(O_{k,S}).$

Comme la fibration $X\to\mathbb{P}^1$ vérifie l'hypothèse (\textsc{Abélienne-Scindée}), on fixe
une composante irréductible $Z_i$ de multiplicité un de la fibre
$X_{P_i}.$ La fermeture algébrique $K_i$ de $k_i$ dans le corps
des fonctions de $Z_i$ est une extension abélienne de $k_i.$ On
peut écrire $K_i/k_i$ comme un composé d'un nombre fini
d'extensions cycliques $K_{i,j}/k_i.$

Comme dans la preuve de la proposition \ref{keyprop1}, lorsqu'on a la fibration $X\to\mathbb{P}^1$
et le sous-ensemble hilbertien généralisé
$\textsf{Hil},$ on choisit un revêtement étale fini galoisien $U'\to U\subset\mathbb{P}^1$ et un modèle entier
$\mathcal{U}'\to\mathcal{U}$ au-dessus de $O_{k,S_1}$ qui se factorise à travers
$\mathcal{V}$ un modèle entier de $U_{k'},$ où $k'$ est une extension finie de
$k$ qui ne dépend que  de $\textsf{Hil}.$
Comme il y a des fibres géométriquement non intègres, on ne peut pas appliquer directement
l'estimation de Lang-Weil, l'ensemble $S_2$ dans la preuve de la proposition \ref{keyprop1} n'existe plus.

Quitte à augmenter $S,$ on peut supposer que $S\supset S_1\cup \Omega^\infty.$
On va redéfinir l'ensemble $I\subset\Omega_k$ qui joue un rôle crucial dans la preuve de la proposition \ref{keyprop1}.

On fixe un caractère primitif $\chi$ du groupe cyclique $Gal(K_{i,j}(\mathbb{P}^1)/k_i(\mathbb{P}^1))=Gal(K_{i,j}/k_i),$
l'élément $(\chi,g'_i)$ du groupe de Brauer $Br(k_i(\mathbb{P}^1))$ est noté simplement par
$(K_{i,j}/k_i,g'_i),$ on note $A_{i,j}=cores_{k_i(\mathbb{P}^1)/k(\mathbb{P}^1)}(K_{i,j}/k_i,g'_i)\in Br(k(\mathbb{P}^1)),$
pour la construction de ces éléments \textit{cf.} \cite{CT-SD} \S 1.1.

Alors $A_{i,j}\in Br(\mathbb{P}^1\setminus D)$ pour tout $i,j,$ où $D$ est un certain ensemble fini de
points fermés de $\mathbb{P}^1$ contenant $D_0.$
Quitte à remplacer $a$ par un multiple, on peut supposer que $a$ annule tous les $A_{i,j}$ et que $a$ est
un multiple de $[k':k].$

On choisit un zéro-cycle effectif global $z_0\in Z_0(X)$ de degré $d_0>0$ tel que $y_0=\pi_*(z_0)$ soit à support hors de
$D\cup\{\infty\}.$

On part d'une famille $\{z_v\}_{v\in\Omega_k}\bot Br_{vert}(X),$
on peut supposer que les $z_v$ sont supportés hors des fibres au-dessus de $D$ d'après le lemme 3.1 de
\cite{Liang1}. Le lemme formel de Harari (\textit{cf.} 2.6.1 de \cite{Harari}; pour la version pour les zéro-cycles,
\textit{cf.} \cite{CT-Sk-SD}, Lemme 4.5) dit, en modifiant $z_v$ pour $v\in S'\setminus S$ si nécessaire, qu'il
existe un sous-ensemble fini $S'\subset\Omega_k$ contenant $S$  tel que
$$\sum_{v\in S'} \langle A_{i,j},\pi_*(z_v) \rangle _v=0.$$

Pour simplifier les notations, on suppose que $S'=S$
et on approxime les zéro-cycles locaux pour $v\in S'=S.$

On utilise le procède de la preuve de la proposition \ref{keyprop1},
pour chaque $v\in S$ on trouve des zéro-cycles effectifs
$z_v^2,\tau_v\in Z_0(X_v)$ de degré assez grand $d\equiv1(mod\mbox{ }ad_0)$
tel que $z_v^2-z_v$ soit $a$-divisible dans $Z_0(X_v),$ tel que $\pi_*(\tau_v)$ soit séparable à support
hors de $D\cup supp(y_0)\cup\{\infty\},$
et tel que $\tau_v$ soit suffisamment proche de $z^2_v.$

On rappelle le lemme suivant.

\begin{lem}[\cite{CT-Sk-SD}, Lemme 1.2]\label{key lemma for loc split}
Soit $k$ un corps de nombres, on note $Spec(O)$ un ouvert non-vide de $Spec(O_k).$ Soit $\Pi:\mathcal{X}\To \mathcal{P}^1$ un morphisme
plat, projectif avec $\mathcal{X}$ régulier et lisse sur $O.$ On note $\pi:X\to\mathbb{P}^1$ la restriction à la fibre
générique de $\Pi.$ Soit $T\subset \mathcal{P}^1$ un sous-schéma fermé fini étale sur $O$ tel que les fibres de $\Pi$ au-dessus
des points qui ne sont pas dans $T$ sont géométriquement intègres.
Soit $T=\bigcup T_i$ la décomposition des composantes irréductibles, soit $k_i$
le corps des fonctions de $T_i.$

Quitte à restreindre l'ouvert $Spec(O)\subset Spec(O_k),$ on a les assertions suivantes.

(a) Étant donné un point fermé $u$ de $\mathcal{P}^1,$
si la fibre $\mathcal{X}_u$ sur le corps fini $k(u)$ est géométriquement intègre, alors
elle contient un $k(u)$-point lisse.

(b) Étant donné un point fermé $\theta$ de $\mathbb{P}^1$ avec son adhérence de
Zariski $Spec(\tilde{O})\simeq \tilde{\theta}\subset\mathcal{P}^1,$
où $\tilde{O}/O$ est fini avec $Frac(\tilde{O})=k(\theta),$
on note $\tilde{O}'$ la clôture intégale de $\tilde{O}$ dans $k(\theta).$
Si $u\in \tilde{\theta}\subset\mathcal{P}^1$ est un point fermé tel
que $\mathcal{X}_u/k(u)$ soit géométriquement intègre,
alors $X_\theta$ contient un
$k(\theta)_v$-point lisse où $v$ est une place de $k(\theta)$
(associée à un point fermé de $Spec(\tilde{O}')$) au-dessus de $u.$

(c) Soit $u$ dans un des $T_i,$ il définit alors une place $v_i$ de $k_i.$
Supposons qu'il existe une composante irréductible $Z$ de la fibre
de $f$ en $T_i\times_Ok=Spec(k_i)$ de multiplicité un. On note $K_i$ la clôture algébrique de $k_i$ dans le corps des fonctions
de $Z.$ Si la place $v_i$ décompose totalement dans l'anneau des entiers $O_i\subset K_i,$ alors $\mathcal{X}_u/k(u)$
%%contient une composante irréductible de multiplicité un qui
est géométriquement intègre.

(d) On suppose que, pour chaque $i,$ il existe au moins une composante
irréductible de $f^{-1}(Spec(k_i))$ de multiplicité un. Alors, étant
donné $M/k$ une extension finie, il existe une extension finie $M'$ de $M,$
telle que pour presque toute place $v\in\Omega_k$
décomposant totalement dans $M'$ on ait l'assertion suivante:

L'application $f:X(L)\To \mathbb{P}^1(L)$ est surjective pour toute extension finie $L$ de $k_v$ .
\end{lem}

\begin{proof}
Voir le lemme 1.2 de Colliot-Thélène/Skorobogatov/Swinnerton-Dyer \cite{CT-Sk-SD} pour la démonstration,
c'est une version pour les points fermés,
comme indiqué dans la page 20 de \cite{CT-Sk-SD} la même preuve fonctionne dans ce cas.

En comparant avec \cite[Lemme 1.2]{CT-Sk-SD}, ici on remplace
\og scindée\fg\mbox{ }(\emph{i.e.} la fibre contient une composante irréductible de multiplicité un qui est géométriquement
intègre) par \og géométriquement intègre\fg\mbox{ }simultanément dans l'hypothèse et dans les conclusions,
la même preuve reste valable.
\end{proof}

On fixe une extension (non-triviale) finie galoisienne $M$ de $k$ qui contient tous les $k_i$ et $K_{i,j},$
d'après le lemme \ref{key lemma for loc split} (d),
il existe une extension finie $M'\supset M\cdot k'$ de $k$
et $I$ un sous-ensemble infini de $\Omega_k$ en dehors de $S$ contenant presque toutes
les places qui sont décomposées totalement dans $M'.$
On a alors, pour toute extension finie $L/k_v$ ($v\in I$), l'application $X(L)\To \mathbb{P}^1(L)$ est surjective,
les places dans $I$ sont décomposées totalement dans $M$ et
dans $k'.$

On fixe l'application (Lemme \ref{E})
$$\gamma:\textbf{E}\To((\Omega_k\setminus S)\cap I)\otimes_kk'$$ comme dans la preuve de la
proposition \ref{keyprop1} et on fixe une place
$v_0\in I\setminus S\setminus\{v_c,c\in\textbf{E}\}$
(pas simplement dans $\Omega_k\setminus S\setminus\{v_c,c\in\textbf{E}\}$).
On a ensuite les $f_v\in k_v(\mathbb{P}^1)^*/k^*_v$ pour toute $v\in S\cup\{v_c,c\in\textbf{E}\}\cup\{v_0\}$
décrites comme suit.

Pour chaque $c\in \textbf{E},$ on trouve, comme dans la preuve de la proposition \ref{keyprop1}, un zéro-cycle
$x_c+x'_c$ séparable effectif de degré $d$ à support hors de $D\cup supp(y_0)\cup\{\infty\}.$
De même, pour $v_0,$ on trouve un point fermé
$x_{v_0}$ de degré $d$ en dehors de $D\cup supp(y_0)\cup\{\infty\}.$
On écrit  $div_{\mathbb{P}^1_v}(f_v)=\pi_*(\tau_v)-d\infty$ pour $v\in S,$
$div_{\mathbb{P}^1_{v_c}}(f_{v_c})=(x_c+x'_c)-d\infty$ pour $c\in \textbf{E},$ et
$div_{\mathbb{P}^1_{v_0}}(f_{v_0})=x_{v_0}-d\infty$ pour $v=v_0.$

Pour toute $v\in S\cup\{v_c,c\in\textbf{E}\}\cup\{v_0\},$ on pose
$\rho_{i,v}=f_v(P_i)\in k_{i,v}^*=(k_i\otimes_kk_v)^*.$
Puisque $I\setminus\{v_c,c\in\textbf{E}\}\setminus\{v_0\}$ est infini,
le théorème de Dirichlet généralisé (\cite{Sansuc82}, Corollaire 4.4) permet de trouver, pour chaque $i,$ un élément
$\rho_i\in k^*_i$ qui soit suffisamment proche de $\rho_{i,v}\in (k_i\otimes_kk_v)^*$ pour
$v\in S\cup\{v_c,c\in\textbf{E}\}\cup\{v_0\}$ et qui soit une unité en dehors de
$(S\cup I)\otimes_kk_i\sqcup\{w_i\},$ où $w_i$ est une place finie de $k_i$ en dehors de $(S\cup I)\otimes_kk_i$ telle que
de plus $w_i(\rho_i)=1.$

Comme $d$ peut être choisi assez grand, d'après le théorème de Riemann-Roch (pour les détails \textit{cf.} les preuves
des lemmes 5.1 et 5.2 de Colliot-Thélène \cite{CT99}), on obtient une fonction
$f\in O_{k,S\cup I}[\mathbb{A}^1]$ telle que $f$ soit suffisamment proche de $f_v$
pour toute $v\in S\cup\{v_c,c\in\textbf{E}\}\cup\{v_0\},$
et telle que $f(P_i)=\rho_i$ pour tout $i.$
En écrivant $div_{\mathbb{P}^1}(f)={\theta}-d\infty,$
on obtient un zéro-cycle effectif $\theta,$ de plus ${\theta}$ est un point fermé
dans $U$ hors de $D\cup supp(y_0)\cup \{\infty\}$ car il est suffisamment proche de $x_{v_0}.$
Puisque le zéro-cycle ${\theta}$ est suffisamment proche de $x_c+x'_c$ pour $c\in \textbf{E},$
le point fermé $\theta\in \textsf{Hil}$ d'après le même argument de la proposition
\ref{keyprop1} (ici on utilise le fait que $d$ et $[k':k]$ sont premiers entre eux).
Le zéro-cycle ${\theta}$ est aussi suffisamment proche de $\pi_*(\tau_v)$ pour $v\in S,$

Il reste à vérifier que ${\theta}$ est déployé localement partout. On suit
principalement l'idée de Colliot-Thélène/Skorobogatov/Swinnerton-Dyer \cite{CT-Sk-SD}.

Pour $w\in \Omega_{k({\theta})}$ au-dessus de $v\in \Omega_k:$

Si $v\in S,$ le théorème des fonctions implicites implique que ${\theta}_v$ est déployé.

Si $v\in I,$ le lemme \ref{key lemma for loc split} (d) implique que $X(k({\theta})_w)\twoheadrightarrow \mathbb{P}^1(k({\theta})_w)$
pour toute $w$ au-dessus de $v.$

Si $v\notin S\cup I,$ on note $\tilde{{\theta}}\simeq Spec(A)$ l'adhérence de Zariski de ${\theta}$ dans $\mathcal{P}^1,$
où $A/O_{k,S}$ est fini avec $Frac(A)=k({\theta}),$ on sait alors que
$O_{k({\theta}),S}$ est la clôture intégale de $A$ dans $k({\theta}).$
On fixe une place $w$ de $k({\theta})$
au-dessus de $v,$ il définit un point fermé $w\in Spec(O_{k({\theta}),S}),$ ce point se trouve
au-dessus d'un certain point fermé $w_{\theta}\in\tilde{{\theta}}.$
On remarque
que $\tilde{{\theta}}$ et $T_i$ sont définis localement par $f$ et $g_i$ respectivement.
Il y a deux cas possibles.

(i) Si $w_{\theta}$ est contenue dans un (et alors un seul) des $T_i.$
On sait que $g_i({\theta})\in k({\theta})$ et $\rho_i=f(P_i)\in O_{k_i,S\cup I}.$
On rappelle que pour $w'\in \Omega_{k_i}\setminus(S\cup I)\otimes_kk_i,$
$w'(\rho_i)=1$ si $w'=w_i,$ $w'(\rho_i)=0$ si $w'\neq w_i.$
Donc, en restreignant au-dessus de $\Omega_k\setminus(S\cup I)\subset Spec(O_{k,S}),$ l'intersection
$T_i\cap\tilde{{\theta}}$ ne contient qu'un point
noté par $w_i\in T_i,$ et le multiplicité d'intersection
de $T_i$ et $\tilde{{\theta}}$ en $w_i$ égale $1$ car $w_i(\rho_i)=1.$ Alors $w_i,$ vu comme un
point fermé $w_{\theta}$ de $\tilde{{\theta}},$ doit être un point régulier de $\tilde{{\theta}}.$ On a alors $w=w_{\theta}=w_i,$
$k_{iw_i}=k({\theta})_w$ et $w(g_i({\theta}))=w_i(\rho_i)=1.$

(ii) Si $w_{\theta}\notin T_i$ pour tout $i,$ alors $\mathcal{X}_{w_{\theta}}/k(w_{\theta})$
est géométriquement intègre par la construction de $T_i,$ on a donc
$X_{\theta}(k({\theta})_w)\neq\emptyset$ d'après le lemme \ref{key lemma for loc split}(b).
On sait que $g_i({\theta})$ est une unité (modulo $w_{\theta}$) dans
$k(w_{\theta})\subset k(w)$ car $g_i({\theta})\notin T_i\cap\tilde{{\theta}},$ donc $w(g_i({\theta}))=0.$

Pour vérifier que ${\theta}$ est déployé localement partout, il reste la place $w$ dans le cas (i) où $w=w_i\in T_i.$
On note $E_i=k_i\otimes_kk({\theta}),$ $F_{i,j}=K_{i,j}\otimes_kk({\theta}).$
On a
$\langle A_{i,j},{\theta} \rangle _{\mathbb{P}^1}=cores_{k({\theta})/k}cores_{E_i/k({\theta})}(F_{i,j}/E_i,g'_i({\theta}))\in Br(k)$
par définition.

On rappelle que
$$\sum_{v\in S} \langle A_{i,j},\pi_*(z_v) \rangle _v=0$$ et
$ \langle A_{i,j},\pi_*(z_v) \rangle _v= \langle A_{i,j},\pi_*(z^2_v) \rangle _v$ pour toute
$v\in S.$
Alors
$$\sum_{v\in S} \langle A_{i,j},\pi_*(\tau_v) \rangle _v=\sum_{v\in S} \langle \pi^*(A_{i,j}),\tau_v \rangle _v=\sum_{v\in S} \langle \pi^*(A_{i,j}),z^2_v \rangle _v=\sum_{v\in S} \langle A_{i,j},\pi_*(z^2_v) \rangle _v=0,$$
donc $$\sum_{v\in S} \langle A_{i,j},{\theta} \rangle _v=0$$
par continuité de l'accouplement de Brauer-Manin
(pour $v\in S,$ ${\theta}$ est suffisamment proche de $\pi_*(\tau_v),$ $\tau_v$ est suffisamment proche de
$z_v^2,$ on remarque que $z_v-z^2_v$ est $a$-divisible, et $a$ annule $A_{i,j}$).

On a donc $$\sum_{v\in \Omega_k\setminus S} \langle A_{i,j},{\theta} \rangle _v=0$$ car ${\theta}$ est global.
Donc
$$\sum_{v\in \Omega_k\setminus S}inv_v(cores_{k({\theta})/k}cores_{E_i/k({\theta})}(F_{i,j}/E_i,g'_i({\theta})))=0,$$
$$\sum_{w\in \Omega_{k({\theta})}\setminus S\otimes_kk({\theta})}inv_w(cores_{E_i/k({\theta})}(F_{i,j}/E_i,g'_i({\theta})))=0.$$

À partir de maintenant on suppose que  $w\in(\Omega_k\setminus S)\otimes_kk(\theta).$

Si $w\in I\otimes_kk(\theta),$ soit $v$ la place de $k$ au-dessous de $w,$
par construction l'extension des corps locaux associés aux $K_{i,j}/k_i$
au-dessus de $v$ est triviale, l'extension $F_{i,j}/E_i$ est alors triviale au-dessus de $w,$
on trouve $$inv_w(cores_{E_i/k({\theta})}(F_{i,j}/E_i,g'_i({\theta})))=0.$$

Si $w\notin I\otimes_kk(\theta)$ et $w\neq w_i$ (plus précisément, le point $w_{\theta}$ associé à $w$ n'est pas dans $T_i$),
on rappelle que dans ce cas $w(g_i(\theta))=0,$ alors $g_i({\theta})$ est une unité en $w,$ d'où
$inv_w(cores_{E_i/k({\theta})}(F_{i,j}/E_i,g'_i({\theta})))=0.$

On trouve finalement $$(\star)\mbox{ }inv_{w_i}(cores_{E_i/k({\theta})}(F_{i,j}/E_i,g'_i({\theta})))=0.$$

On considère la flèche  $E_i\To E_i\otimes_{k({\theta})}k({\theta})_{w_i},$ où $E_i\otimes_{k({\theta})}k({\theta})_{w_i}$
est un produit d'extensions de $k({\theta})_{w_i}.$
En remarquant que $w(N_{E_i/k({\theta})}(g'_i({\theta})))=w(g_i({\theta}))$ égale soit $0$ si $w\neq w_i$
soit $1$ si $w=w_i,$ il y a donc seulement une de ces extensions, notée par $E_{i,w},$
dans laquelle l'image de $g'_i({\theta})$ n'est pas
une unité mais une uniformisante, de plus, $E_{i,w}/k({\theta})_{w_i}$ est triviale.
L'égalité ($\star$) implique que $(F_{i,j}/E_i,g'_i({\theta}))\otimes_{E_i}E_{i,w}=0,$
on a alors pour tout $j$ l'extension cyclique $K_{i,j}/k_i$ est triviale après $\otimes_{E_i}E_{i,w}$
car $g'_i({\theta})$ est une uniformisante de $E_{i,w},$ on trouve que
$K_i/k_i$ est triviale après $\otimes_{E_i}E_{i,w}.$
D'après le lemme \ref{key lemma for loc split}(c), la réduction $\mathcal{X}_{w_i}/k(w_i)$
de $X_\theta$ modulo $w_i$ est géométriquement intègre,
$X_{\theta}$ contient donc un $k({\theta})_{w_i}$-point d'après
le lemme \ref{key lemma for loc split}(b).
\end{proof}

%%% ----------------------------------------------------------------------

\section{Cas général}\label{Casgeneral}

%%% ----------------------------------------------------------------------

\subsection{Fibrations au-dessus d'une courbe de genre quelconque}\label{Cas1}

Dans cette sous-section, on considère le théorème principal (sauf l'exactitude de la suite $(E)$) pour le cas où
$B=C$ est une courbe lisse de groupe $\sha(Jac(C))$ fini. Dans ce cas, l'obstruction de Brauer-Manin associée au sous-groupe
$Br_{vert}(X)$ suffit, \emph{i.e.}
les conclusions deviennent respectivement:
pour les zéro-cycles de degré $1$ sur $X$

(1) l'obstruction de Brauer-Manin associée au groupe $Br_{vert}(X)$ est la seule au principe de Hasse;

(2) l'obstruction de Brauer-Manin associée au groupe $Br_{vert}(X)$ est la seule à l'approximation faible;

(3) l'obstruction de Brauer-Manin associée au groupe $Br_{vert}(X)$ est la seule à l'approximation forte.

\noindent Ce cas est une généralisation des théorèmes principaux 1.3 et 1.4 de Wittenberg
\cite{Wittenberg} au sens que
$\textsf{Hil}$ est un sous-ensemble hilbertien généralisé au lieu d'un ouvert dense de $C.$

Dans la section \S \ref{suiteE-C}, on va expliquer comment établir la suite $(E)$ pour $X,$
d'où la conclusion (3) est automatiquement vérifiée (\emph{cf.} \S\ref{suiteE}), de plus la même méthode
démontre également (1) et (2).

%%% ----------------------------------------------------------------------

\subsection{Fibrations au-dessus de $\mathbb{P}^n$}\label{Cas2}
Dans cette sous-section, on explique comment on peut montrer
le théorème principal (sauf l'exactitude de la suite $(E)$) dans le cas où
$B=\mathbb{P}^n$ est l'espace projectif. En répétant la stratégie de Wittenberg
\cite[Théorème 3.4]{WittenbergLNM}, le résultat a été presque établi dans \cite[Théorème 3.5]{Liang2},
la seule différence est de remplacer l'ouvert dense $U\subset\mathbb{P}^n$ par un
sous-ensemble hilbertien généralisé $\textsf{Hil}\subset\mathbb{P}^n.$
Il suffit de vérifier que
la condition liée à le sous-ensemble hilbertien généralisé se comporte bien dans la récurrence, qui est le lemme
suivant.
Ce lemme a été mentionné dans \cite{Harari3} \S 1.2 sans preuve, on présente une preuve pour le confort du lecteur.

\begin{lem}\label{lemhilbertien}
Soient $k$ un corps de nombres et $\textsf{Hil}\subset\mathbb{A}^{r+s}$ un sous-ensemble hilbertien généralisé.
Alors $p_1(Hil)$ contient un certain sous-ensemble hilbertien généralisé de $\mathbb{A}^r,$ où
$p_1:\mathbb{A}^{r+s}=\mathbb{A}^r\times\mathbb{A}^s\to\mathbb{A}^r$ est la projection canonique.
\end{lem}

\begin{proof}
On note $p_2:\mathbb{A}^{r+s}=\mathbb{A}^r\times\mathbb{A}^s\to\mathbb{A}^s$ la deuxième projection.
Soit $\textsf{Hil}$ défini par $V\buildrel\rho\over\twoheadrightarrow U\subset\mathbb{A}^{r+s}$ où $U$
est un ouvert non vide et où $\rho$ est un morphisme étale fini.
La projection $W_2=p_2(U)$ est un ouvert non vide de $\mathbb{A}^s,$ il existe un ouvert non vide $W_2'\subset W_2$
tel que pour tout point fermé $\theta_2\in W_2',$ $U_{\theta_2}=U\cap p_2^{-1}(\theta_2)$ soit lisse sur $k(\theta_2),$
la variété $V_{\theta_2}$ est alors lisse.
Comme $\mathbb{A}^{r+s}(k)\cap \textsf{Hil}$ est Zariski dense dans $\mathbb{A}^{r+s},$
\textit{cf.} \cite{Ekedahl}, on trouve que $p_2(\mathbb{A}^{r+s}(k)\cap \textsf{Hil})\cap W_2'\neq\emptyset,$
il existe alors un $k$-point $\theta_2\in p_2(\textsf{Hil})\cap W_2'.$
L'ouvert $U_{\theta_2}$ de $p_2^{-1}(\theta_2)$ est alors non vide lisse,
on définit $Z=\rho^{-1}(U_{\theta_2}),$ c'est une variété lisse.
Par construction, il existe un point fermé $\theta_1\in\mathbb{A}^r$ tel que
$(\theta_1,\theta_2)$ soit contenu dans $\textsf{Hil},$ \textit{i.e.} $\rho^{-1}(\theta_1,\theta_2)$ est connexe,
la variété $Z$ est alors aussi connexe, donc intègre. Puisque $\theta_2$ est un $k$-point, le morphisme
$Z\to U_{\theta_2}$ définit un sous-ensemble hilbertien généralisé de $\mathbb{A}^r\times\theta_2\simeq \mathbb{A}^r,$
contenu dans $p_1(\textsf{Hil}).$
\end{proof}

%%% ----------------------------------------------------------------------

\section{La suite exacte $(E)$}\label{suiteE}

Dans cette section, on explique la suite $(E)$ et établit son exactitude pour les fibrations
considérées dans le théorème principal.

Tout d'abord, on note $A_0(X)=ker[deg:CH_0(X)\to\Z].$
La théorie du corps de classes global implique que l'image diagonale de $A_0(X)$ dans $\prod_{v\in\Omega_k}A_0(X_v)$ est
contenue dans le noyau de l'accouplement de Brauer-Manin. Ceci donne un complexe
$$A_0(X)\to\prod_{v\in\Omega_k} A_{0}(X_v)\to Hom(Br(X),\Q/\Z),$$
où la deuxième flèche est induite par l'accouplement de Brauer-Manin.
On considère le complexe complété
$$(E_0)\mbox{   }A_0^{\chapeau}(X)\to \prod_{v\in\Omega_k}A_0^\chapeau(X_v)\to Hom(Br(X),\Q/\Z),$$
où $\widehat{M}$ désigne $\varprojlim_{n>0}M/n$ pour un groupe abélien $M.$

L'exactitude de la suite $(E_0)$ a été conjecturée par Colliot-Thélène/Sansuc \cite{CTSansuc81},
Conjectures A, C, pour les surfaces rationnelles; par Kato/Saito \cite[\S 7]{KatoSaito86},  et par
Colliot-Thélène \cite[Conjecture 1.5]{CT95}, pour toutes les variétés lisses.

La formulation suivante est due à van Hamel \cite{vanHamel}, développée  par Wittenberg \cite{Wittenberg}.
On définit $CH_{0,\mathbb{A}}(X)$ comme le groupe
$$\prod_{v\in\Omega_k^\textmd{f}}CH_0(X_v)\times\prod_{v\in\Omega_k^\infty}Coker[N_{\bar{k}_v/k_v}:CH_0(\overline{X}_v)\to CH_0(X_v)].$$
De même, on trouve un complexe
$$(E)\mbox{   }CH_0^{\chapeau}(X)\to CH_{0,\A}^\chapeau(X)\to Hom(Br(X),\Q/\Z),$$
et on espère que $(E)$ soit exact pour toute variété lisse et géométriquement connexe $X.$
Il est remarqué par Wittenberg que
l'exactitude de $(E)$ implique l'exactitude de $(E_0),$
et implique que l'obstruction de Brauer-Manin est la seule au principe de Hasse pour les zéro-cycles de degré
$1$ sur $X,$ \textit{cf.} \cite{Wittenberg}, Remarques 1.1 (ii)(iii).
L'exactitude de $(E)$ implique aussi que l'obstruction de Brauer-Manin est la seule à l'approximation faible et
forte pour les zéro-cycles de degré $1,$ \emph{cf.} \cite[Proposition 2.2.1]{Liang4} et sa preuve.
Réciproquement, c'est possible mais ce n'est pas évident que l'approximation forte pour les
zéro-cycle de degré $1$ implique l'exactitude de $(E)$ qui concerne les zéro-cycles de tout degré.

On remarque que, pour les places réelles, le conoyau $Coker[N_{\mathbb{C}/\mathbb{R}}:CH_0(X_\mathbb{C})\to CH_0(X_\mathbb{R})]$
est calculé par Colliot-Thélène/Ischebeck \cite{CT-I}.

\bigskip

Dans le reste de cette section, on va établir l'exactitude de la suite $(E)$ pour les fibrations considérées dans le théorème
principal, on suppose que $X_\eta$ est rationnellement connexe à partir de maintenant.
Dans \S\ref{suiteE-C}, pour le cas où $B=C$ est une courbe,
on peut l'établir directement en utilisant l'argument de Wittenberg
\cite{Wittenberg}, de plus la suite $(E)$ reste exacte même si l'on remplace $Br(X)$ par $Br_{vert}(X).$
Dans \S\ref{suiteE-P}, pour le cas où $B=\mathbb{P}^n,$ l'argument utilise la conclusion (2) dans le théorème principal
qui est montrée dans \S\ref{Cas2}.

\subsection{Le cas $B=C$}\label{suiteE-C}
Dans son article \cite{Wittenberg}, Wittenberg établit l'exactitude de la suite $(E)$ pour une fibration
$X\to C$ satisfaisant (\textsc{Abélienne-Scindée}), en supposant que

- pour presque tout point fermé $\theta$ de $C,$ pour tout entier $n>0$ et tout ensemble fini
$S\subset\Omega_{k(\theta)},$ l'image de $CH_0(X_\theta)$ dans $\prod_{w\in S}CH_0((X_\theta)_w)/n$
contient l'image de $\prod_{w\in S}X_\theta(k(\theta)_w)$ dans ce même groupe.

\noindent Cette hypothèse est vérifiée si $X_\theta$ satisfait l'approximation faible pour les points rationnels
ou pour les zéro-cycles de degré $1.$

On rappelle l'idée de la preuve dans \cite{Wittenberg}.
On part d'une fibration $X\to C,$ on construit un morphisme dominant
$\psi:C\to\mathbb{P}^1$ et on considère la fibration $W'\to\mathbb{P}^1,$
où $W'$ est une compactification lisse de la restriction à la Weil $W\to\mathbb{P}^1$ de $X\to C$ le long de $\psi.$
L'exactitude de $(E)$ pour $X$ se déduit de l'exactitude de $(E)$ pour  $W',$ cette réduction a été fait
dans \S 4.3 de \cite{Wittenberg}. Afin d'établir $(E)$ pour $W',$
on applique le théorème 4.8 de \cite{Wittenberg} (une variante de \cite[théorème 4.1]{CT-Sk-SD})
à la fibration $W'\to \mathbb{P}^1.$

Dans \cite{Wittenberg}, l'hypothèse arithmétique est posée sur $X_\theta$ pour presque tout point fermé
$\theta$ de $C.$ Par contre, ici l'hypothèse arithmétique est posée seulement pour  $\theta\in \textsf{Hil}$
où $\textsf{Hil}$ est un sous-ensemble hilbertien généralisé de $C.$
Pour que le même argument fonctionne pour notre situation, il suffit d'adapter la notion
de sous-ensemble hilbertien généralisé avec cette procédure de réduction et avec \cite[Théorème 4.8]{Wittenberg}.

D'abord, on suppose que $\textsf{Hil}$ est défini par $Z\to U\subset C,$ son composé avec
$\psi$ définit alors un sous-ensemble hilbertien généralisé $\textsf{Hil}_1$ de $\mathbb{P}^1$ tel que
pour tout $\theta_1\in \textsf{Hil}_1$ on ait automatiquement $\theta=\psi^{-1}(\theta_1)\in \textsf{Hil}.$
Donc les sous-ensembles hilbertiens généralisés
se comporte bien avec la réduction dans \S 4.3 de \cite{Wittenberg}.

Ensuite, on considère le théorème 4.8 de \cite{Wittenberg}, avec l'hypothèse arithmétique
supposée seulement sur un sous-ensemble hilbertien généralisé $\textsf{Hil}_1\subset\mathbb{P}^1.$
Remarquons que l'assertion dans \S \ref{Casparticulier2} généralise \cite[théorème 4.1]{CT-Sk-SD}
et de plus l'argument fonctionne également pour $\{z_v\}_{v\in\Omega_k}$
de degré premier à un nombre entier bien choisi, ceci confirme la conclusion de \cite[Théorème 4.8]{Wittenberg}
sous l'hypothèse sur $\textsf{Hil}_1.$

Enfin, on remarque que, dans tout cet argument, on utilise seulement le sous-groupe
$Br_{vert}(X)$ de $Br(X).$

\subsection{Le cas $B=\mathbb{P}^n$}\label{suiteE-P}
Puisque le cas $n=1$ est contenu dans \S \ref{suiteE-C}, on suppose que $n>1.$
Comme expliqué dans \cite[Théorème 3.5]{Liang2}, on applique la stratégie
de Wittenberg dans la preuve du \cite[Théorème 3.4]{WittenbergLNM} qu'on rappelle comme suit.
On part d'une fibration $X\to\mathbb{P}^n,$
la variété $X$ est vue comme une fibration $X\to\mathbb{P}^{n-1}$ à fibres géométriquement
intègres. De plus, l'obstruction de Brauer-Manin est la seule pour les zéro-cycles de degré
$1$ sur les fibres $X_\theta$ si $\theta$ est contenu dans un certain sous-ensemble
hilbertien généralisé de $\mathbb{P}^{n-1}.$
En appliquant le théorème 3.5 de \cite{Liang2}, on conclut que l'obstruction de Brauer-Manin
est la seule à l'approximation faible pour les zéro-cycles de degré $1.$ En plus,
sans modifier ni les hypothèses ni les arguments, la même
conclusion est valable pour les zéro-cycles de degré $\delta$ quelconque sur $X.$

D'après Wittenberg \cite[Prop 3.1]{Wittenberg}, afin d'établir $(E)$ pour
une fibration au-dessus d'une courbe, il suffit de vérifier une propriété $(P_S)$
pour tout ensemble fini $S\subset\Omega_k.$ On remarque que la même conclusion reste valable
si l'on remplace la base par l'espace projectif. Tous les arguments de Wittenberg fonctionnent.
En plus, la propriété $(P_S)$ et les arguments deviennent plus simples car
l'application induite $CH_0(X)\to CH_0(\mathbb{P}^{n-1})=\mathbb{Z}$ est simplement l'application de
degré.

Il reste à vérifier la propriété pour $X\to\mathbb{P}^{n-1}:$
\begin{enumerate}
\item[$(P_S)$]
Soit $\{z_v\}_{v\in\Omega_k}$ une famille de zéro-cycles de degré $\delta.$ Si elle est orthogonale à
$Br(X),$ alors pour tout entier $n>0,$  il existe un zéro-cycle $z$ de $X$ de degré $\delta,$ tel que
pour toute $v\in S$ on ait $z=z_v$ dans $CH_0(X_v)/n$ si $v$ est finie et
$z=z_v+N_{\bar{k}_v/k_v}(u_v)$ dans $CH_0(X_v)$ pour un $u_v\in CH_0(\overline{X_v})$
si $v$ est réelle.
\end{enumerate}
Cette propriété est impliquée par la propriété suivante.
\begin{enumerate}
\item[$(P'_S)$]
Soit $\{z_v\}_{v\in\Omega_k}$ une famille de zéro-cycles de degré $\delta.$ Si elle est orthogonale à
$Br(X),$ alors pour tout entier $n>0,$  il existe un zéro-cycle $z$ de $X$ de degré $\delta,$ tel que
pour toute $v\in S$ on ait $z=z_v$ dans $CH_0(X_v)/2n.$
\end{enumerate}
\noindent Cette dernière propriété est exactement l'assertion que l'obstruction
de Brauer-Manin est la seule à l'approximation faible pour les zéro-cycles
de degré $\delta,$ qui est vérifiée par $X.$
Ceci établit l'exactitude de $(E)$ pour le cas $B=\mathbb{P}^n$ dans le théorème principal.

%%% ----------------------------------------------------------------------

\section{Fibré en surfaces de Châtelet}\label{Fibreenchatelet}

La notion de sous-ensemble hilbertien généralisé nous permet d'appliquer le résultat principal à certaines
fibrations, par exemple, les fibrations en surfaces de Châtelet (ou encore plus général) au-dessus d'une courbe
sous l'hypothèse (\textsc{Abélienne-Scindée}).

D'abord, on rappelle la notion de $p$-fold de Châtelet.
Soient $K$ un corps et $L$ une extension finie de degré $n.$
On fixe une $K$-base linéaire $w_1,\ldots,w_n$ de $L.$
Soit $P(z)\in K[z]$ un polynôme de degré $dn$ où $d>0$ est un entier.
L'équation normique
$$N_{L/K}(x_1w_1+\cdots+x_nw_n)=P(z)$$ est un polynôme dans $K[x_1,\ldots,x_n,z],$
elle définit une variété lisse et géométriquement intègre dans $\mathbb{A}^{n+1}.$
Il existe un modèle projectif lisse $X=X_{L/K,P(z)}$ de cette variété (\cite{anthony-bianca}, Proposition 2.1).
Un tel modèle est appelé un $p$-fold de Châtelet si $L/K$ est une extension cyclique de degré premier $p$ et si $d=2,$
c'est une surface de Châtelet si $p=2.$
Si $K$ est un corps de nombres,
la variété $X$ vue comme un fibré en coniques au-dessus de $\mathbb{P}^1$ via l'indéterminée $z,$
vérifie que l'obstruction de Brauer-Manin est la seule au principe de Hasse/à l'approximation faible pour
les zéro-cycles de degré $1$
(Théorème principal, ou \cite[Théorème 1.3]{Wittenberg}).

\textbf{Fait.} Si $K$ est un corps de nombres,
le groupe de Brauer $Br(X)$ est égal à $im[Br(K)\to Br(X)]$ lorsque $P(z)$ est un polynôme irréductible sur $K.$
Par conséquent, le principe de Hasse et l'approximation faible pour les zéro-cycles de degré $1$ sont valables pour
$X=X_{L/K,P(z)}.$

Ce fait se déduit de la proposition 2.5 de Colliot-Thélène/Harari/Skorobogatov \cite{CTHSk}, voir aussi
le corollaire 3.3 de V\'arilly-Alvarado/Viray \cite{anthony-bianca}.
Récemment, Wei obtient un résultat plus général, une fois que
$L/K$ est une extension cyclique (pas nécessairement de degré premier) et $P(z)$ est
irréductible sur $K$ (de degré quelconque), on a encore le fait ci-dessus, \textit{cf.} \cite{Wei}, Théorème 1.4.

Soit $k$ un corps de nombres.
On considère une fibration au-dessus d'une courbe $\pi:X\to C,$ dont
la fibre générique est définie par l'équation
$N_{L/K}(x_1w_1+\cdots+x_pw_p)=P(z),$ où $P(z)\in K[z]$ est un polynôme supposé irréductible
sur $K=k(C)$ et où $L$ est une extension cyclique de $K.$
Le polynôme $P(z)$ définit une extension finie $K'$ de $K,$ à laquelle il associe une courbe
normale $Z$ et un morphisme fini dominant $Z\to C.$ Ce morphisme est étale au-dessus d'un ouvert non vide, il
définit alors un sous-ensemble hilbertien généralisé $\textsf{Hil}\subset C$ tel que pour tout $\theta\in \textsf{Hil}$
la fibre $X_\theta$ vérifie le principe de Hasse et l'approximation faible pour les zéro-cycles de degré $1$
(le fait ci-dessus).
D'après le théorème principal,
si $\sha(Jac(C))$ est fini et si $\pi$ vérifie (\textsc{Abélienne-Scindée}), l'obstruction
de Brauer-Manin est la seule au principe de Hasse et à l'approximation faible pour les zéro-cycles de degré
$1$ sur $X,$ la suite $(E)$ est exacte pour $X$ (avec la finitude de $\sha(Jac(C))$ supposée).

Il reste à vérifier (\textsc{Abélienne-Scindée}). Dans le cas particulier où
le terme gauche $$N_{L/k(C)}(x_1w_1+\cdots+x_pw_p)$$ ne varie pas le long de $C,$ on peut vérifier directement
l'hypothèse (\textsc{Abélienne-Scindée}). Plus précisément, dans ce cas $L=l(C)$ où $l/k$ est une extension cyclique
de degré $p,$ et $w_1,\ldots,w_p$ est une $k$-base linéaire de $l$ (donc c'est une $K$-base linéaire
de $L$ car $C$ est une courbe géométriquement intègre sur $k$). Soit $\theta$ un point fermé
de $C,$ la fibre $X_\theta,$ définie par $N_{l\otimes_kk(\theta)/k(\theta)}(x_1w_1+\cdots+x_pw_p)=P_\theta(z),$

- est géométriquement intègre, \textit{a fortiori} satisfait (\textsc{Abélienne-Scindée}),
si $P_\theta(z)\in k(\theta)[z]$ est un polynôme non nul;

- satisfait (\textsc{Abélienne-Scindée}), si $P_\theta(z)$ est un polynôme nul.
En fait, la fibre $X_\theta$ a $p$ composantes irréductibles (toute de multiplicité un)
après l'extension abélienne $l(\theta)/k(\theta),$
dont chacune est géométriquement intègre sur $l(\theta).$

En résumé, on trouve:

\begin{prop}\label{application}
Soient $k$ un corps de nombres et $l$ une extension cyclique. Soit $K$ une extension finie
de $k(t).$ Supposons que $\sha(Jac(C))$ est fini, où $C$ est la courbe projective lisse et géométriquement intègre
de corps des fonctions $K.$ Soit $X$ une variété projective, lisse, géométriquement intègre, et
$k$-rationnellement équivalente à la $k$-variété définie par l'équation (en variables $(x_1,\ldots,x_p,z)$)
$$N_{l(C)/k(C)}(x_1w_1+\cdots+x_pw_p)=P(z),$$
où $P(z)\in K[z]$ est un polynôme irréductible sur $K.$

Alors, la suite $(E)$ est exacte pour $X.$
\end{prop}

\begin{rem}
Si l'extension $L/k(C)$ ne provient pas simplement d'une extension finie $l/k,$
\textit{a priori}, on ne sait pas si, pour tout modèle $X\to C$ de
$X_\eta/k(C),$ les fibres satisfont l'hypothèse\footnote{Par contre,
toutes les fibres du solide de Poonen \cite{Poonen} sont géométriquement intègres,
dans ce cas-là, le théorème principal de \cite{Liang1} suffit à conclure. Les solides de Poonen sont aussi dans le cadre
de la proposition \ref{application}.} (\textsc{Abélienne-Scindée}).
Généralement, on part d'un $p$-fold de Châtelet $Y/k(C)$ défini par
$N_{L/k(C)}(x_1w_1+\cdots+x_pw_p)=P(z),$ et on veut
trouver un modèle $X\to C$ à fibre générique $X_\eta$ isomorphe à $Y$ sur $k(C),$
tel que les fibres fermées satisfont (\textsc{Abélienne-Scindée}).
Comme $\bar{k}(C)$ est un $C_1$-corps d'après le théorème de Tsen,
la variété $Y$ admet toujours un
$\bar{k}(C)$-point car elle est définie par un polynôme homogène,
l'homogénéisation de $N_{L/k(C)}(x_1w_1+\cdots+x_pw_p)=P(z),$
de degré $p$ en $p+1$ variables $x_0,x_1,\ldots,x_p.$
D'où on obtient une $\bar{k}$-section de $\pi_{\bar{k}}$ pour n'importe quel modèle $\pi:X\to C$ de $Y/k(C),$
alors toute fibre $X_\theta$ possède une composante irréductible de multiplicité un.
Mais on ne sait pas si $Y/k(C)$ admet un modèle
$X\to C$ vérifiant la condition d'abélianité de (\textsc{Abélienne-Scindée}).
\end{rem}

\smallskip
\small \noindent \textbf{Remerciements.}
Je tiens à remercier D. Harari pour ses nombreuses discussions très utiles
pendant la préparation de ce travail, et pour son aide pour le
français. Je remercie O. Wittenberg de sa patiente explication
de son travail récent \cite{Wittenberg} et de ses commentaires sur la première version de ce texte.
Je remercie également J.-L. Colliot-Thélène pour ses suggestions.
% et pour ses
%commentaires pertinents après la lecture de la première version
%de ce travail.

\normalsize

% ------------------------------------------------------------------------
%GATHER{Xbib.bib}   % For Gather Purpose Only
%GATHER{Paper.bbl}  % For Gather Purpose Only

\bibliographystyle{plain}
\bibliography{astuce}
\end{document}